\documentclass[a4paper]{article}
\usepackage[most]{tcolorbox}
\usepackage{lineno}
\usepackage{amsmath,amsthm,amsfonts,amssymb,mathrsfs} % mathrsfs enables \mathscr
\usepackage{lipsum}
\usepackage{mathtools}
\usepackage{dsfont}
\usepackage{pdflscape}
\usepackage{float}
\usepackage{mathabx}
\usepackage{tikz-cd}
\usepackage{bm}
\usepackage{tikz}
\usepackage{pgfplots}
\usetikzlibrary{arrows}
\usepackage{appendix}
\usepackage{hyperref}

\usepackage{authblk} % for author and email information
\usepackage[margin=1in]{geometry}

\usepackage{cite}   % to cite ranges of references instead of long lists of reference numbers

\usepackage{graphicx}
\usepackage{color}
\usepackage{mathrsfs}

\newtheorem{theorem}{Theorem}

\newtheorem{corollary}{Corollary}

\def\Ce{\mathds{C}}

\def\Ne{\mathds{N}}

\def\Re{\mathds{R}}

\def\1e{\mathds{1}}

\def\vgf{\mathbf{f}}
\def\vgg{\mathbf{g}}

\def\vg0{\mathbf{0}}

\def\B{\mathscr{B}}

\def\D{\mathscr{D}}

\def\F{\mathscr{F}}

\def\M{\mathscr{M}}

\def\ggg{\boldsymbol{\gamma}}

\def\ran#1{\mathop{\mathrm{ran}}#1}
\def\ker#1{\mathop{\mathrm{ker}}#1}

\def\meas{\mathop{\mathrm{meas}}}

\def\ud{\,\mathrm{d}}

\def\vphan{\vphantom{\bigl|}}

\def\scal#1#2{\langle #1,#2 \rangle}
\def\scalc#1#2{
    \left\langle\vphan
    \mskip 2mu #1 \mskip 2mu,
    \mskip 2mu #2 \mskip 2mu
    \right\rangle
    }

\def\set#1#2{\{\mskip 1mu #1 \mskip 1mu
    | \mskip 1mu #2 \mskip 1mu \}
    }
\def\setc#1#2{
    \left\{
    \mskip 2mu #1 \mskip 2mu
    \left| \vphan\vphantom{#1#2} \right.
    \mskip 2mu #2 \mskip 2mu
    \right\}
    }

\def\mod#1{|\mskip 1mu #1 \mskip 1mu|}
\def\modc#1{
    \left|
    \mskip 2mu #1 \vphan \mskip 2mu
    \right|
    }

\def\norm#1{\| \mskip 1mu #1 \mskip 1mu \|}
\def\normc#1{
    \left\|
    \mskip 2mu #1 \vphan \mskip 2mu
    \right\|
    }

\def\operator#1#2#3#4#5{
        \begin{array}{lcll}
        \displaystyle #1 \colon & \displaystyle #2 & \longrightarrow & \displaystyle #3 \\[.2ex]
                     & \displaystyle #4 & \longmapsto     & \displaystyle #5
        \end{array}
        }
\title{Interpolating between Tikhonov regularization and spectral cutoff}
\author[1]{Martin Sæbye Carøe}
\author[1]{Mirza Karamehmedovi\'c}
\author[2]{Pierre Maréchal}
\date{}
\affil[1]{Department of Applied Mathematics and Computer Science,
Technical University of Denmark, DK-2800 Kgs. Lyngby, Denmark}
\affil[2]{Universit\'e Toulouse III - Paul Sabatier, 118 route de Narbonne,
FR-31062 Toulouse, France}

\begin{document}
%\linenumbers
\maketitle

\begin{abstract}
Regularizing a linear ill-posed operator equation
can be achieved by manipulating the spectrum of the
operator's pseudo-inverse. Tikhonov regularization
and spectral cutoff are well-known techniques within
this category. This paper introduces an interpolating
formula that defines a one-parameter family of
regularizations, where Tikhonov and spectral cutoff
methods are represented as limiting cases. By adjusting
the interpolating parameter taking into account the specific operator equation under consideration, it is possible to mitigate
the limitations associated with both Tikhonov and
spectral cutoff regularizations. The proposed approach
is demonstrated through numerical simulations in the
fields of signal and image processing. 
\end{abstract}

%\begin{enumerate}
%    \item generalize Theorem 4.9 p. 103 in Colton-Kress: relaxed conditions on $q$ in the regularization scheme, only requiring something at the well-known singular values $\mu_n$ instead of over the whole range $0<\mu\le\|A\|$
%    \item exploit BandWidth to find the best Tikhonov parameter $\lambda$
%\end{enumerate}

%%%%%%%%%%%%%%%%%%%%%%
\section{Introduction}
\label{section:introduction}

We consider a general linear ill-posed operator
equation
\begin{equation}
\label{ill-posed_operator_equation}
Tf=g,
\end{equation}
in which $T\colon F\to G$ is a bounded linear
operator from $F$ to $G$, where $F$ and $G$ are infinite
dimensional separable Hilbert spaces. In a number of real world
applications, $T$ is such that
\begin{equation}
\label{ill-posedness}
\inf\setc{\norm{Tf}\vphantom{\big|}}{\,f\in(\ker{T})^\perp,\;\norm{f}=1}=0.
\end{equation}
When~$T$ is injective, an assumption which
will be in force throughout, Condition~\eqref{ill-posedness}
boils down to
\begin{equation}
\label{ill-posedness-injective}
\inf\setc{\norm{Tf}\vphantom{\big|}}{\norm{f}=1}=0.
\end{equation}

The latter condition results in ill-posedness, meaning that:
\begin{enumerate}
\item[(i)]
the range $\ran{T}$ of~$T$ is not closed in~$G$;
\item[(ii)]
the densely defined pseudo-inverse $T^\dagger\colon \ran{T}+\ker{T^*}\to F$
is unbounded, so that the minimum-norm least squares solution to the linear
equation $Tf=g$ does not depend continuously on the data~$g$.
\end{enumerate}

The purpose of regularization theory is to provide approximate solutions
to~\eqref{ill-posed_operator_equation} that depend continuously on the
data, so as to avoid instability in the inversion process. See~\cite{arridge} and the references therein for a thorough treatment of regularization. 
The main objective of this paper is to bridge the gap between two
of the most frequently used regularization methods, namely the Tikhonov
regularization technique and the so-called spectral cutoff method. We shall propose
a family of regularization techniques depending on some {\sl interpolating
parameter} that encompasses both Tikhonov regularization and the spectral cutoff
as special (extreme) cases. The proposed filters can be naturally included in learning schemes~\cite{arridge} for further, data-driven refinement.

%%%%%%%%%%%%%%%%%%%%%%%%%%%%%%%%%%%%%%%%%%%%
\section{A family of regularization schemes}

%===================================
\subsection{Spectral regularization}

Prior to defining our interpolating family of regularization schemes,
we establish a general theorem on spectral methods (see Theorem~\ref{spectral-reg} below).
This theorem is already well-known, but our approach to its proof is new in that we make use of the {\sl singular
value expansion} of general bounded operators. More precisely, we shall make use of the following theorem.

\begin{theorem}\sf
\label{sve}
Let $T\colon F\to G$ be an injective bounded linear operator, where $F$ and $G$
are real Hilbert spaces. There then exist
\begin{enumerate}
\item
a Borel space $(\M,\B,\mu)$ with completely separable topology,
\item
an unitary operator $V\colon L^2(\M,\B,\mu)\to F$ and
an isometry $W\colon L^2(\M,\B,\mu)\to G$,
\item
an essentially bounded measurable function $\sigma\colon\M\to\Re$
that is strictly positive $\mu$-almost everywhere,
\end{enumerate}
such that $T=W[\sigma]V^*$, in which $[\sigma]$ denotes the operator of
multiplication by~$\sigma$. Moreover, $\ran{W}=\overline{\ran{T}}$. 
\end{theorem}
See~\cite[Theorem 3]{crane2020singular}.
The following diagram illustrates the singular value expansion
exhibited in the last theorem:
$$
\begin{tikzcd}
F \arrow{r}{T}
& G \\
L^2(\M,\B,\mu)\arrow[swap]{u}{V}
\arrow{r}{[\sigma]}
& L^2(\M,\B,\mu)\arrow[swap]{u}{W}
\end{tikzcd}
$$

Under Condition~\eqref{ill-posedness-injective}, $\ran{T}$
is not closed. By~\cite[Theorem 4]{crane2020singular}, this
implies that~$\sigma$ is not bounded away from zero.
On the other hand, since $\ran{W}=\overline{\ran{T}}$,
$$
\ker{T^*}=(\ran{W})^\perp=\ker{W^*}.
$$
Clearly, the injectivity of~$T$ implies that of~$W$,
which in turn implies that $W^\dagger=(W^*W)^{-1}W^*=W^*$.
As can be easily seen, the following factorizations hold:
\begin{eqnarray*}
T^*&=&V[\sigma]W^*,\\
T^*T&=&V[\sigma^2]V^*,\\
T^\dagger=(T^*T)^{-1}T^*&=&V[\sigma^{-1}]W^*.
\end{eqnarray*}

Recall that a family $(R_\alpha)$ of bounded operators from~$G$
to~$F$ is said to be a {\sl regularization of~$T^\dagger$}
if there exists a {\sl parameter choice rule}
$$
\operator{\alpha}{\Re_+\times G}{\Re_+\setminus\{0\}}
{(\delta,g^\delta)}{\alpha(\delta,g^\delta)}
$$
such that
\begin{enumerate}
\item[(1)]
$\sup\setc
{\norm{\alpha(\delta,g^\delta)}}
{g^\delta\in G,\;\norm{g^\delta-g}\leq\delta}
\to 0$ as $\delta\downarrow 0$;
\item[(2)]
$\sup\setc
{\norm{R_{\alpha(\delta,g^\delta)}g^\delta-T^\dagger g}}
{g^\delta\in G,\;\norm{g^\delta-g}\leq\delta}
\to 0$ as $\delta\downarrow 0$;
\end{enumerate}
Recall also that $(R_\alpha)$ is a regularization of~$T^\dagger$
if $(R_\alpha)$ converges pointwise to~$T^\dagger$ on its domain
$$
\D(T^\dagger)=\ran{T}+(\ran{T})^\perp=\ran{T}+\ker{T^*}.
$$
See~\cite[Proposition 3.4]{engl1996regularization}.

The following result is well-known; however, we here propose a novel proof:
\begin{theorem}\sf
\label{spectral-reg}
Let~$T\colon F\to G$ be an injective bounded operator with
singular value expansion $T=W[\sigma]V^*$ as in Theorem~\ref{sve}.
Assume that $\sigma$ is not bounded away from zero, so that the
inverse problem $Tf=g$ is ill-posed.
Let $q\colon\Re_+^*\times\Re_+\to\Ce$ be such that
\begin{enumerate}
\item[(A1)]
for every $(\alpha,\sigma)\in\Re_+^*\times\Re_+$, 
$\mod{q(\alpha,\sigma)}\leq M$ for some positive constant~$M$,
\item[(A2)]
for every $\sigma\in\Re_+$, $\mod{q(\alpha,\sigma)}\leq c_\alpha\sigma$,
in which $c_\alpha$ is some positive constant.
\end{enumerate}
Then $R_\alpha:= V[\sigma^{-1}q(\alpha,\sigma)]W^*$ is well-defined and bounded,
and $\norm{R_\alpha}\leq c_\alpha$.
Moreover, if
\begin{enumerate}
\item[(A3)]
for every $\sigma\in\Re_+$,
$\lim_{\alpha\downarrow 0}q(\alpha,\sigma)=1$,
\end{enumerate}
then $(R_\alpha)$ is a regularization of~$T^\dagger$.
\end{theorem}

\begin{proof}
Let~$g\in\D(T^\dagger)$ and let $\bar{g}$ be its projection onto the
closure of the range of~$T$. Observe that, since~$g$
belongs to the domain of~$T^\dagger$, its projection~$\bar{g}$
must actually belong to the range of~$T$. Hence $\bar{g}=Tf$
for some $f\in F$, from which we deduce that
$$
[\sigma^{-1}]W^*\bar{g}=
[\sigma^{-1}]W^*W[\sigma]V^*f=
V^*f\in L^2(\M,\B,\mu).
$$
By Assumption (A1), $[q(\alpha,\sigma)][\sigma^{-1}]W^*g=[q(\alpha,\sigma)]V^*f$
also belongs to $L^2(\M,\B,\mu)$, which in turn implies that
$$
V[q(\alpha,\sigma)][\sigma^{-1}]W^*g=:R_\alpha g
$$
is well-defined.
Under Assumption (A2), we have, for every $g\in G$,
\begin{eqnarray*}
\normc{R_\alpha g}^2
&=&
\normc{V[\sigma^{-1}q(\alpha,\sigma)]W^*g}^2\\
&=&
\normc{[\sigma^{-1}q(\alpha,\sigma)]W^*g}^2\\
&\leq&
\normc{\sigma^{-1}q(\alpha,\sigma)}_\infty^2\normc{W^*g}^2\\
&=&
\normc{\sigma^{-1}q(\alpha,\sigma)}_\infty^2\normc{g}^2\\
&\leq&
c_\alpha^2\normc{g}^2,
\end{eqnarray*}
in which the second equality is due to the unitarity of~$V$
and the third equality is due to the fact that~$W$ is an isometry
(so that $W^*W=I$, where~$I$ denotes the identity).
The first assertion in the theorem follows.

Next, let $f^\dagger:= T^\dagger g$ and $f_\alpha:=R_\alpha g$.
From the definition of~$R_\alpha$ and the formula for~$T^\dagger$,
we readily see that
$$
f^\dagger-f_\alpha=
V\big[\sigma^{-1}(1-q(\alpha,\sigma))\big]W^*g=
V\big[1-q(\alpha,\sigma)\big][\sigma^{-1}]W^*\bar{g}=
V\big[1-q(\alpha,\sigma)\big]V^*f.
$$
Therefore, using the unitarity of~$V$, 
$$
\normc{f^\dagger-f_\alpha}^2=
\normc{\big[1-q(\alpha,\sigma)\big]V^*f}_{L^2(\M,\B,\mu)}^2=
\int\modc{1-q(\alpha,\sigma)}^2 \big(V^*f(\sigma)\big)^2\ud\mu(\sigma).
$$
Finally, under Assumption (A3), the function
$\sigma\mapsto \modc{1-q(\alpha,\sigma)}^2$ converges pointwise
to zero, and Lebesgue's dominated convergence theorem then implies
that $f_\alpha\to f^\dagger$ as $\alpha\downarrow 0$.
\end{proof}

Notice that in the above theorem, $q$ is 
allowed to take complex values. Notice also that $q(\alpha,\cdot)$
need only be defined on the range of~$\sigma$.

The latter theorem generalizes Theorem 4.9 in~\cite{colton2013inverse},
in which compactness of~$T$ was assumed. This special case now appears
as a corollary. If we let~$T$ be a compact operator, then
the singular value expansion takes the form
$$
Tx=
\sum_{n=1}^\infty\sigma_n(g_n\otimes f_n)f=
\sum_{n=1}^\infty\sigma_n\scalc{f_n}{f}g_n,
$$
in which the $\sigma_n$ are the so-called singular values, $(f_n)$
is a Hilbert basis of~$F$ and $(g_n)$ is a Hilbert basis of $\overline{\ran{T}}$.
This is a particular case of Theorem~\ref{sve}, in which
the Borel space is $l^2_\Re$, the space of square summable real sequences, endowed with
the counting measure. It can be easily checked that the operators~$W$ and~$V$
are respectively given in this case by
$$
W(s_n)=\sum_{n=1}^\infty s_n g_n
\quad\hbox{and}\quad
V(s_n)=\sum_{n=1}^\infty s_n f_n.
$$
Theorem~\ref{spectral-reg} then takes the following form:

\begin{corollary}\sf
\label{spectral-reg-compact}
Let~$T\colon F\to G$ be an injective compact operator with
singular value expansion $T=\sum_{n=1}^\infty\sigma_n(g_n\otimes f_n)$.
Let $q\colon\Re_+^*\times\set{\sigma_n}{n\in\Ne^*}\to\Ce$ be such that
\begin{enumerate}
\item[(A1)]
for every $\alpha\in\Re_+^*$ and every $n\in\Ne^*$, 
$\mod{q(\alpha,\sigma_n)}\leq M$ for some positive constant~$M$,
\item[(A2)]
for every $n\in\Ne^*$, $\mod{q(\alpha,\sigma_n)}\leq c_\alpha\sigma_n$,
in which $c_\alpha$ is some positive constant.
\end{enumerate}
Then $R_\alpha:=\sum_{n=1}^\infty\sigma_n^{-1}q(\alpha,\sigma_n)(f_n\otimes g_n)$
is well-defined and bounded,
and $\norm{R_\alpha}\leq c_\alpha$.
Moreover, if
\begin{enumerate}
\item[(A3)]
for every $n\in\Ne^*$,
$\lim_{\alpha\downarrow 0}q(\alpha,\sigma_n)=1$,
\end{enumerate}
then $(R_\alpha)$ is a regularization of~$T^\dagger$.
\end{corollary}

An important example of an ill-posed inverse problem
is the standard deconvolution problem.
Given a function $\gamma$ in~$L^1(\Re^n)$,
we consider the convolution operator
$T_\gamma: L^2(\Re^n) \to L^2(\Re^n)$ defined by 
\begin{equation*}
    T_\gamma f = \gamma * f.
\end{equation*}
The deconvolution problem consists in solving $T_\gamma f=g$ for~$f$
when~$g$ is empirically known. This problem is ubiquitous in many areas
of applied sciences, including signal and image processing,
physics, statistics etc.
If $\gamma$ has an almost everywhere positive-valued Fourier
transform~$\hat{\gamma}$ then the singular value expansion is explicit:
$$
T_\gamma=W[\hat{\gamma}]V^*
\quad\hbox{with}\quad
W=V=U^{-1}
$$
where $U\colon L^2(\Re^n)\to L^2(\Re^n)$ is the
Fourier-Plancherel operator. Recall that the Fourier-Plancherel
operator can be defined as the closure to~$L^2(\Re)$ of the
Fourier transform, denoted likewise, on $L^1(\Re^n)\cap L^2(\Re^n)$.
Our definition of the Fourier transform of an integrable function~$f$
on~$\Re^n$ is:
$$
(Uf)(\xi)=
\widehat{f}(\xi):=
\int_{\Re^n} f(x)e^{-2\pi i\scal{x}{\xi}}\ud x.
$$
Under the standard assumption that $\meas\set{\xi\in\Re^n}{\hat{\gamma}(\xi)=0}=0$
the operator $T_\gamma$ is injective.
Remember that by the Riemann-Lebesgue lemma, $\widehat{\gamma}$ is continuous and
vanishes at infinity, which implies (via~\cite[Theorem 4]{crane2020singular})
that the deconvolution problem is ill-posed.
This particular situation is illustrated below:
$$
\begin{tikzcd}
L^2(\Re^n) \arrow{r}{T_\gamma}
& L^2(\Re^n) \\
L^2(\Re^n)\arrow[swap]{u}{U^{-1}}
\arrow{r}{[\hat\gamma]}
& L^2(\Re^n)\arrow[swap]{u}{U^{-1}}
\end{tikzcd}
$$

The last example reveals potential drawbacks of both Tikhonov regularization
and the spectral cutoff. Recall that the variational formulation
of the Tikhonov regularization consists in the minimization of
the functional
$$
\F(f)=\normc{g-T_\gamma f}^2+\normc{f}^2=
\normc{\hat{g}-\hat{\gamma}\hat{f}}^2+\normc{\hat{f}}^2,
$$
where we used Parseval's identity. We see that the penalty term
acts with equal strength everywhere in the Fourier domain. This
can't be optimal since the {\sl low frequencies} are constrained
by the data, $\hat{\gamma}$ being close to~1 near the origin (Charibdis).
On the other hand, the spectral cutoff multiplies the Fourier
transform of~$f$ by a {\sl mask}, which is equivalent to convolving~$f$
by a sinc-like function, thereby creating the Gibbs phenomenon (Sylla).

Our intuition is that we may somehow find a reasonable compromise
placing us in between these two choices for~$q$. In the next subsection,
we propose an interpolation formula for the function~$q$, and show that it fits
the regularization framework of Theorem~\ref{spectral-reg}.
Later on, we shall determine optimal values for the interpolating parameter. 
Finally, we shall illustrate the performance of our new
family of regularization schemes by means of 
numerical simulation.

%==============================================================
\subsection{Interpolating between Tikhonov and spectral cutoff}

Two important special cases of the function $q$ that appears
in Theorem~\ref{spectral-reg} are
$$
q_0(\alpha,\sigma)=
\frac{\sigma^2}{\alpha+\sigma^2}
\quad\hbox{and}\quad
q_\infty(\alpha,\sigma)=
\begin{dcases}
1 &\hbox{if }\sigma>\sqrt{\alpha},\\
1/2 &\hbox{if }\sigma=\sqrt{\alpha},\\
0 &\hbox{if }\sigma<\sqrt{\alpha}.
\end{dcases} 
$$
The function $q_0$ corresponds to Tikhonov regularization.
It satisfies (A1), (A2) and (A3) with $M=1$ and
$c_\alpha=\frac{1}{2\sqrt{\alpha}}$.
The function $q_\infty$ corresponds to spectral cutoff.
It satisfies (A1), (A2) and (A3) with $M=1$ and
$c_\alpha=\frac{1}{\sqrt{\alpha}}$.

We now define the family $(q_\tau)_{\tau\in\Re_+}$ by
\begin{equation}
\label{interpolating-q}
q_\tau(\alpha,\sigma)=
\frac{1}{1+\left(\frac{\sqrt{\alpha}}{\sigma}\right)^{2+\tau}}.
\end{equation}
It is readily seen that $\tau=0$ yields the Tikhonov case
while in the limit $\tau\to\infty$ one retrieves the spectral
cutoff case.

As can be seen in Figure \ref{fig_filter_functions}, the filter functions $q_\tau$ operate somewhere in-between the Tikhonov regularization and TSVD by making the cut-off around $\sqrt{\alpha}$ more steep than for the Tikhonov regularization, but less steep than the TSVD.

\begin{figure}[H]
\begin{centering}
\begin{tikzpicture}

    \node[inner sep=0pt] (plot1) at (0,0)
        {\includegraphics[width=7cm]{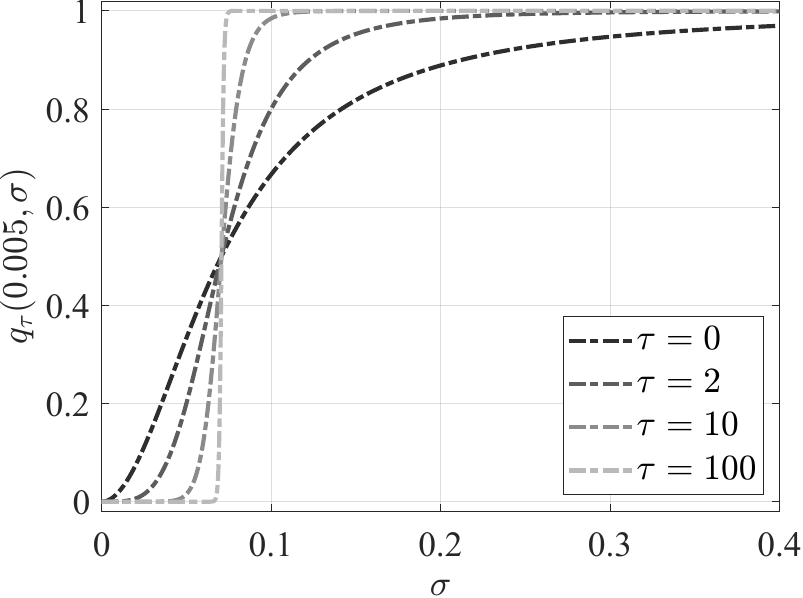}};
        
    \node[inner sep=0pt] (plot2) at (8,0)
        {\includegraphics[width=7cm]{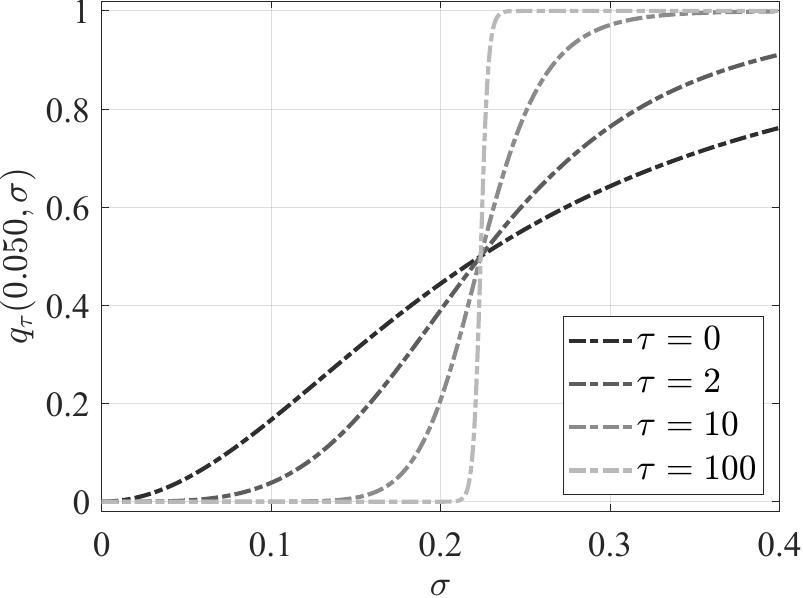}};

\end{tikzpicture}
    \caption{A plot of $q_\tau$ with $\alpha=0.005$ (left) and $\alpha = 0.05$ (right) and four different values of $\tau$.}\label{fig_filter_functions}
\end{centering}
\end{figure}

The above function $q_\tau$ corresponds to the minimization of the functional
\begin{equation}
\label{interpolating-q-var}
f\mapsto
\frac{1}{2}\normc{g-Tf}^2+\frac{1}{2}\normc{H_\alpha f}^2,
\end{equation}
in which $H_\alpha$ is the closed operator on~$F$ given by
$$
H_\alpha=V\left[\sqrt{\alpha}\left(\frac{\sqrt{\alpha}}{\sigma}\right)^{\tau/2}\right]V^*.
$$
The closedness of $H_{\alpha}$ guarantees the existence and uniqueness of solution of the minimization problem~\eqref{interpolating-q-var}.~\cite[Theorem 1 on p. 3]{morozov1984}

As a matter fact, using the expansion $T=W[\sigma]V^*$,
we then have (see~\cite[Section 25, p. 214]{morozov1984}):
\begin{eqnarray*}
f_\alpha
&=&
(T^*T+H_\alpha^*H_\alpha)^{-1}T^*g\\
&=&
\left(
V[\sigma^2]V^*+
V\left[\alpha\left(\frac{\sqrt{\alpha}}{\sigma}\right)^\tau\right]V^*
\right)^{-1}
V[\sigma]W^*g\\
&=&
V\left[\frac{1}{\sigma}
\frac{\sigma^2}
{\sigma^2+\alpha\left(\frac{\sqrt{\alpha}}{\sigma}\right)^\tau}
\right]W^*g\\
&=&
V\left[\frac{1}{\sigma}q_\tau(\alpha,\sigma)\right]W^*g,
\end{eqnarray*}
with $q_\tau$ as in~\eqref{interpolating-q}.
We stress here that the operator~$H_\alpha$ depends
both on~$\alpha$ and the function~$\sigma$. Its closedness
is established in Appendix~\ref{closedness-H}.
In a number of cases, including the deconvolution
problem considered in the previous subsection, $\sigma$ is explicitly 
known, so that the dependence of~$H_\alpha$ on~$\sigma$ is not an obstacle.
We also emphasize that, strictly speaking, the above regularization
does not pertain to the so-called {\sl generalized Tikhonov} class.
As a matter of fact, the dependence of~$H$ on~$\alpha$
is not that one would have if~$\alpha$ was merely a weight
in front of some $\alpha$-independent quadratic penalty
in the objective functional displayed in~\eqref{interpolating-q-var}.

%%%%%%%%%%%%%%%%%%%%%%%%%%%%
\section{Numerical examples}

We test the family of regularization schemes on a 1D
deconvolution problem. Recall that the operator
$T_\gamma\colon L^2(\Re) \to L^2(\Re)$ of
convolution by~$\gamma\in L^1(\Re)$ is given
by $T_\gamma = \gamma * f$. It can be diagonalized
by the Fourier-Plancherel operator.

Let $\vgf,\ggg\in\Re^{n}$ with $n\in \mathbb{N}$ be samples of functions $f,\gamma$ in $L^2(\Re)$
and $L^1(\Re)$, respectively. We assume that
the support of~$f$ and~$\gamma$ is contained in
the interval $[-1,1]$, and that the samples are taken at a uniform grid in this box. We can compute the discrete convolution to obtain $\vgg\in\Re^{n}$, given by
\begin{equation}
\label{discrete_convolution}
\begin{split}
    g_j &= \sum_{j'=1}^n f_{j-j'}\gamma_{j'}h,
 \end{split}
\end{equation}
where $\vgf$ and $\ggg$ are extended by zero to
all indices $j\notin \{1,\hdots,n\}$, and where $h$
is the spacing between points on the grid.

Note that the discrete convolutions
\eqref{discrete_convolution} correspond to the
Riemann sum approximations to the continuous
convolution.

Consider a noisy measurement $\vgg\in \Re^{n}$.
We compute reconstructions $\vgf_{\alpha,\tau}\in \Re^{n}$
by computing
\begin{equation}
\label{reconstruction_formula}
    \vgf_{\alpha,\tau} = U^{-1}\left(\frac{1}{U\ggg}
    q_{\tau}(\alpha,U\ggg)\right) U\vgg.
\end{equation}
Here $U: \Re^{n}\to \Re^{n}$ is a linear operator
that approximates the continuous Fourier transform.
The discrete Fourier transform operator can be used
to approximate the Fourier transform, as shown
in \cite{bailey}.

For the 2D example, we will consider the Shepp-Logan
phantom, a piece-wise constant grey-scale image with
values between $0$ and $1$. We convolve $f$ with a
kernel that is equal to the characteristic function
times a scalar multiple. In particular,
$\gamma=c\cdot\chi_A$, where $c>0$ is chosen such
that $\int_A \gamma(x)\ud x = 1$. In the numerical example,
$A = [-s_{\rm blur}, s_{\rm blur}]$, for some number
$s_{\rm blur}>0$. 

%======================
\subsection{1D example}

We consider three different functions
$f^i\,:\, \mathbb{T} \to [0,1]$, $i=1,2,3$,
where $\mathbb{T} = [-1,1]$ is the 1-dimensional torus,
which identifies the elements $-1$ and $1$.

The first function $f^1$ is continuous and piecewise smooth.
The function $f^2$ has one discontinuity. The function $f^3$
is $C^\infty$. All three functions are shown in the top left
of Figures~\ref{fig:1d_1}, \ref{fig:1d_2} and~\ref{fig:1d_3},
respectively.

The (exact) convolutions $g^i=\gamma*f^i$ are computed and
sampled on a uniform grid $x_j=-1+2j/N$, $j=0,1,\hdots, N-1$
to obtain the exact measurements $g^i_j = g^i(x_j)$.
Random Gaussian i.i.d. noise, $\varepsilon_j^\sigma$
with standard deviation $\sigma>0$ is added to obtain
a noisy measurement $\vgg^{i,\sigma}=[g^i_j+\varepsilon_j^\sigma]$.
In the implementation the "exact" measurement $g_j$ is in
fact computed using a discrete convolution on a much
finer grid than the grid $(x_j)$. As explained,
reconstructions are computed using
Equation~\eqref{reconstruction_formula},
where $U$ is an approximation to the Fourier transform
using the discrete Fourier transform.

The error in the reconstruction, relative to the true signal $f^i$, is computed by
\begin{equation*}
    \|\vgf_{\alpha, \tau}^i - \vgf^i\|_2/\|\vgf^i\|_2,
\end{equation*}
where $\vgf^i = [f^i(x_j)]$.

The values of $\alpha$ are chosen in two different ways:
\begin{itemize}
    \item
    The {\sl Morozov principle} \cite{morozov1984}:
    $\alpha$ is chosen as
    \begin{equation*}
        \alpha=\sup\{a'>0\,:\, 
        \|\ggg * \vgf_{\alpha,\tau} - \vgg^\sigma\|_2\leq
        1.1 \cdot \mathbb{E}(\|\varepsilon\|_2)\}.
    \end{equation*}
    where the expected value of the norm of the error
    is $\mathbb{E}(\|\varepsilon\|_2)=\sigma \sqrt{N}$. 
    \item
    Optimal choice: $\alpha$ is found by running through
    a large collection of values and choosing the value of $\alpha>0$
    that gives the lowest reconstruction error
    $\|\vgf_{\alpha, \tau}^i-\vgf^i\|_2/\|\vgf^i\|_2$.
\end{itemize}

The results of the experiment are shown in
Figure~\ref{fig:1d_1}, \ref{fig:1d_2}, \ref{fig:1d_3}
and Table~\ref{tab:1}, \ref{tab:2} and~\ref{tab:3}.

\begin{figure}[H]
    \centering
    \includegraphics[width = 0.9\textwidth]{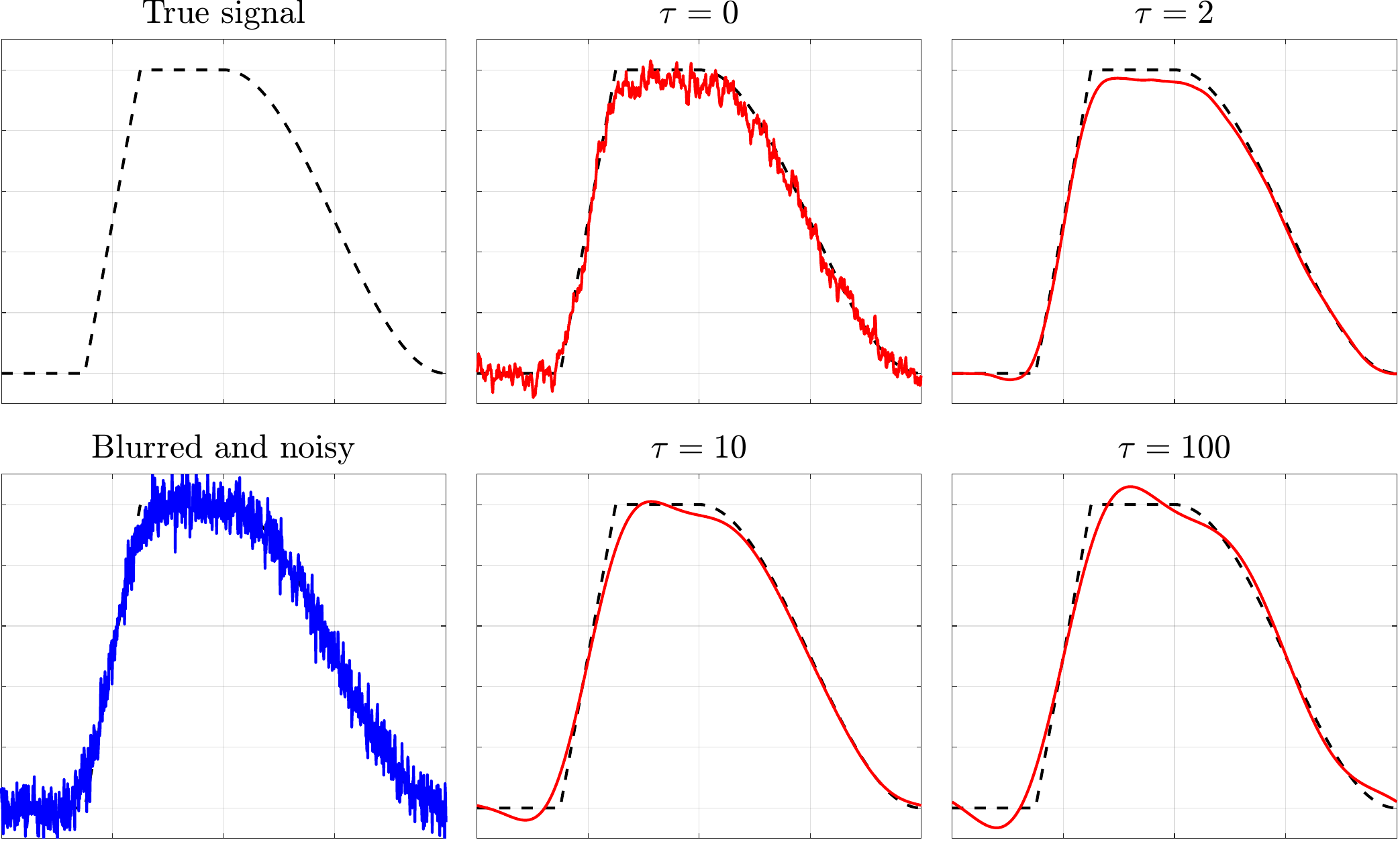}
    \caption{Reconstruction of $f^1$.
    Here we set $s_{\rm blur} = 0.1$, $N=1001$ and noise
    level $\sigma = 0.05$, and
    $\alpha$ is determined using the Morozov principle.}
    \label{fig:1d_1}
\end{figure}
%%%%%%%%%%%%%%%%%%%%%%%%%%%%%%%%%%%%%%%%%%%%%%%%%%%%%%%%%%%%%%%%%
\begin{figure}[H]
    \centering
    \includegraphics[width = 0.9\textwidth]{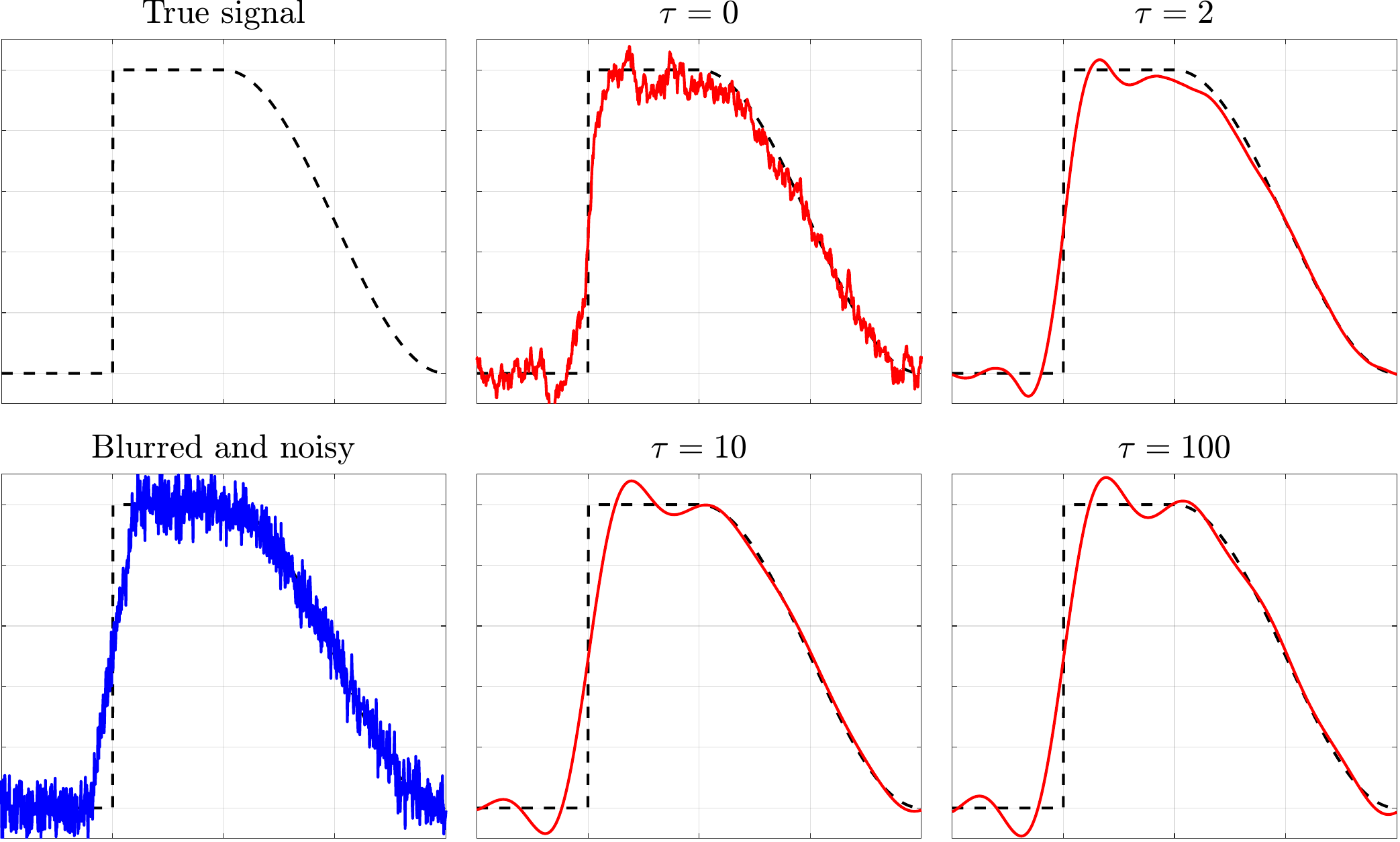}
    \caption{Reconstruction of $f^2$.
    Here we set $s_{\rm blur} = 0.1$, $N=1001$ and noise
    level $\sigma = 0.05$, and
    $\alpha$ is determined using the Morozov principle.}
    \label{fig:1d_2}
\end{figure}

%%%%%%%%%%%%%%%%%%%%%%%%%%%%%%%%%%%%%%%%%%%%%%%%%%%%%%%%%%%
\begin{figure}[H]
    \centering
    \includegraphics[width = 0.9\textwidth]{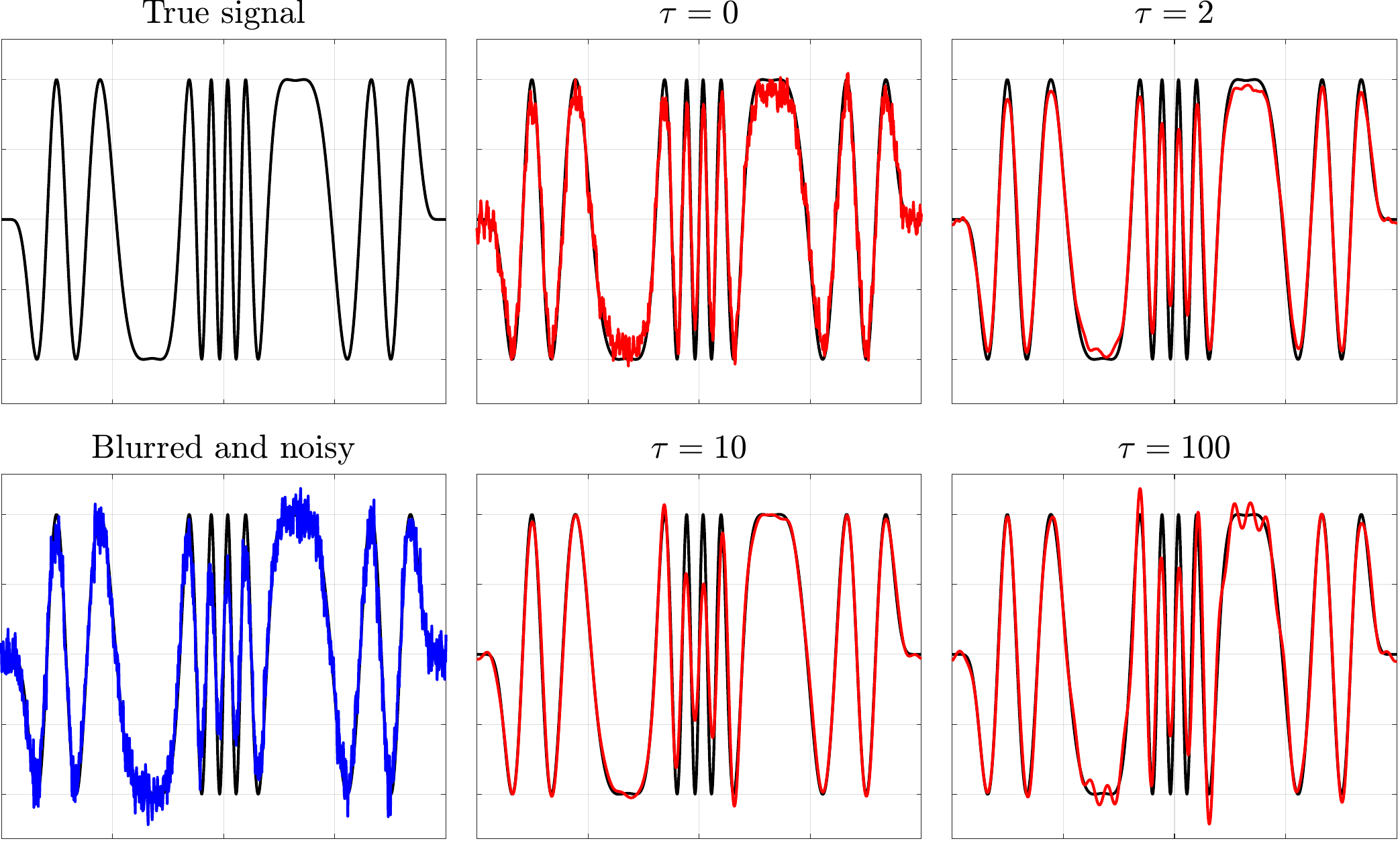}
    \caption{Reconstruction of $f^3$.
    Here we set $s_{\rm blur} = 0.03$, $N=1001$ and noise
    level $\sigma = 0.075$, and
    $\alpha$ is found using the Morozov principle.}
    \label{fig:1d_3}
\end{figure}

%%%%%%%%%%%%%%%%%%%%%%%%%%%%%%%%%%%%%%%%%%%%%%%%%%%%%%%%%%%%%%%%%%
\begin{table}[H]
\centering
\begin{tabular}{llllll}
                             &                              & $\tau = 0$ & $\tau = 2$ & $\tau = 10$ & $\tau = 100$ \\ \hline
\multicolumn{1}{l|}{Morozov} & \multicolumn{1}{l|}{Error}   &           $0.0615$   &   $0.0364$   &   $0.0471$   &   $0.0559$               \\
\multicolumn{1}{l|}{principle}  & \multicolumn{1}{l|}{$\alpha$} &            $0.0354$   &   $0.1727$   &   $0.5032$   &   $0.5620$            \\ \hline
\multicolumn{1}{l|}{Optimal} & \multicolumn{1}{l|}{Error}   &            $0.0607$   &   $0.0192$   &   $0.0193$   &   $0.0194$             \\
\multicolumn{1}{l|}{choice}  & \multicolumn{1}{l|}{$\alpha$} &           $0.0417$   &   $0.0753$   &   $0.1239$   &   $0.1286$            \\ \hline
\end{tabular}
\caption{Reconstruction of $f^1$. Here we set $s_{\rm blur} = 0.1$, $N=1001$ and noise level $\sigma = 0.05$.} \label{tab:1}
\end{table}

%%%%%%%%%%%%%%%%%%%%%%%%%%%%%%%%%%%%%%%%%%%%%%%%%%%%%%%%%%%%%%%%%%%%
\begin{table}[H] 
\centering
\begin{tabular}{llllll}
                             &                              & $\tau = 0$ & $\tau = 2$ & $\tau = 10$ & $\tau = 100$ \\ \hline
\multicolumn{1}{l|}{Morozov} & \multicolumn{1}{l|}{Error}   &           $0.1154$   &   $0.1288$   &   $0.1403$   &   $0.1413$               \\
\multicolumn{1}{l|}{principle}  & \multicolumn{1}{l|}{$\alpha$} &            $0.0328$   &   $0.1490$   &   $0.2841$   &   $0.2497$            \\ \hline
\multicolumn{1}{l|}{Optimal} & \multicolumn{1}{l|}{Error}   &            $0.1154$   &   $0.1037$   &   $0.1081$   &   $0.1098$             \\
\multicolumn{1}{l|}{choice}  & \multicolumn{1}{l|}{$\alpha$} &           $0.0317$   &   $0.0192$   &   $0.0236$   &   $0.0258$            \\ \hline
\end{tabular}
\caption{Reconstruction of $f^2$. Here we set $s_{\rm blur} = 0.1$, $N=1001$ and noise level $\sigma = 0.05$.}\label{tab:2}
\end{table}

%%%%%%%%%%%%%%%%%%%%%%%%%%%%%%%%%%%%%%%%%%%%%%%%%%%%%%%%%%%%%%%
\begin{table}[H] 
\centering
\begin{tabular}{llllll}
                             &                              & $\tau = 0$ & $\tau = 2$ & $\tau = 10$ & $\tau = 100$ \\ \hline
\multicolumn{1}{l|}{Morozov} & \multicolumn{1}{l|}{Error}   &           $0.1585$   &   $0.1355$   &   $0.1621$   &   $0.1821$               \\
\multicolumn{1}{l|}{principle}  & \multicolumn{1}{l|}{$\alpha$} &           $0.0873$   &   $0.2277$   &   $0.3354$   &   $0.3354$             \\ \hline
\multicolumn{1}{l|}{Optimal} & \multicolumn{1}{l|}{Error}   &            $0.1577$   &   $0.0861$   &   $0.0779$   &   $0.0777$             \\
\multicolumn{1}{l|}{choice}  & \multicolumn{1}{l|}{$\alpha$} &           $0.0811$   &   $0.0993$   &   $0.1012$   &   $0.0889$             \\ \hline
\end{tabular}
\caption{Reconstruction of $f^3$. Here we set $s_{\rm blur} = 0.03$, $N=1001$ and noise level $\sigma = 0.075$.}\label{tab:3}
\end{table}

We clearly see that Tikhonov regularization ($\tau = 0$),
does not remove all the oscillations of the noise.
On the other hand, with $\tau=100$, meaning almost
spectral cutoff, we see large oscillations close to
discontinuities or regions with rapid changes in the
function value. Depending on the method for
choosing $\alpha$, choosing an intermediate value
of $\tau$ (for example $\tau=2$ or $\tau=10$) somewhat
removes these two drawbacks of Tikhonov regularization
and spectral cutoff.

Regardless of the value of~$\tau$, the regularization
methods fails to reconstruct the function accurately
at discontinuities or in regions with rapid changes
in the function value.

%======================
\subsection{2D example}

We consider the multi--frequency inverse source problem in two dimensions.
Let $D_0,D\subset\mathbb{R}^2$ be concentric disks with radii $R_0,R$ satisfying $R_0<R$, and let $D_0$ include the support
of an unknown acoustic or electromagnetic source function $s\in L^2(D_0)$. For each wavenumber ('frequency') $k>0$,
the field $u_k$ radiated by $s$ satisfies the Helmholtz equation
\begin{equation}
(\Delta + k^2)u_k(x)=s(x), \qquad x\in\mathbb{R}^2,
\end{equation}
together with the Sommerfeld radiation condition
\begin{equation}
\lim_{|x|\to\infty}\sqrt{|x|}
\big(\partial_{|x|}-ik\big)u_k(x)=0,
\quad \text{uniformly for } x/|x|\in S^1 .
\end{equation}

The measured data are given by the restriction of $u_k$ to the boundary
$\partial D$. The corresponding forward operator $F_k:s\mapsto U_k:=u_k|_{\partial D}$, compact from $L^2(D_0)$ to $L^2(\partial D)$, is given by the radiation integral
\begin{equation}\label{eqn:fo}
U_k(x)=(F_ks)(x)=\int_{y\in D_0}\Phi_k(x-y)s(y)dy,\quad x\in\partial D,
\end{equation}
where $\Phi_k(x)=\frac{i}{4}H_0^{(1)}(k|x-y|)$ is the outgoing fundamental solution of the
Helmholtz equation in the plane, and $H_0^{(1)}$ is the Hankel function of order zero and of first kind. A singular value expansion of $F_k$ was given and characterized in~\cite{bandwidth}; see also~\cite{Bao-2010}. For a set of frequencies $Q=\{k_j\}_{j\in I}$ and corresponding boundary
measurements $\{U_{k_j}\}$, the multi--frequency inverse source problem is to
reconstruct $s\in L^2(D_0)$ such that
\begin{equation}
U_{k_j}=F_{k_j}s,\qquad \forall k_j\in Q.
\end{equation}
The inverse source problem is ill-posed since $\ker F_k=(\Delta+k^2)H^2(D_0)$.~\cite{ziol} The degree of ill--posedness of the inverse problem depends on the
choice of the frequency set $Q$~\cite{Karamehmedovic_2018}, with broader frequency coverage leading to a larger subspace of source functions that can be
reconstructed in a stable manner.

%, with piecewise constant basis functions for the source $s$.
For the numerical experiment, we discretize the domain $D$ using a triangular mesh, generate a synthetic ground truth source $s$, and compute boundary data using the forward operator~\eqref{eqn:fo}. For this experiment we select the frequencies $k_j$
according to
\begin{equation}
k_jR_0 = j\pi, \qquad j=2,\dots,30.
\end{equation}
% This choice is arbitrary, and we should nok make it sound like this is specificly well suit. What is R0? D0 is not defined as a disc here, but maybe it should be.
We assemble the individual-frequency discretized forward operators into a single joint system matrix and compute its singular value
decomposition. Finally, we reconstruct the source using the spectral filtering method defined in \eqref{reconstruction_formula}. Our code is available at \url{https://github.com/msaca-okse/tikhonov_ISP}. 

Figure~\ref{fig:111} shows the ground truth (top left) and reconstructions corresponding to different values of $\tau$.

Figure~\ref{fig:222} shows the singular value spectrum of the combined forward operator (all used frequencies) together with the projections of the ground truth onto the corresponding right singular vectors. We clearly see that neither the Tikhonov regularization ($\tau=0$) nor the truncated SVD regularization ($\tau$ large) give the best solution of the inverse source problem, in that the first includes speckle-like artifacts and the latter exhibits the Gibbs phenomenon near the piecewise constant component of the source. In fact, the choice $\tau=3$ seems to be better than both of the extreme options. 

\begin{figure}
    \centering
    \includegraphics[width=0.75\linewidth]{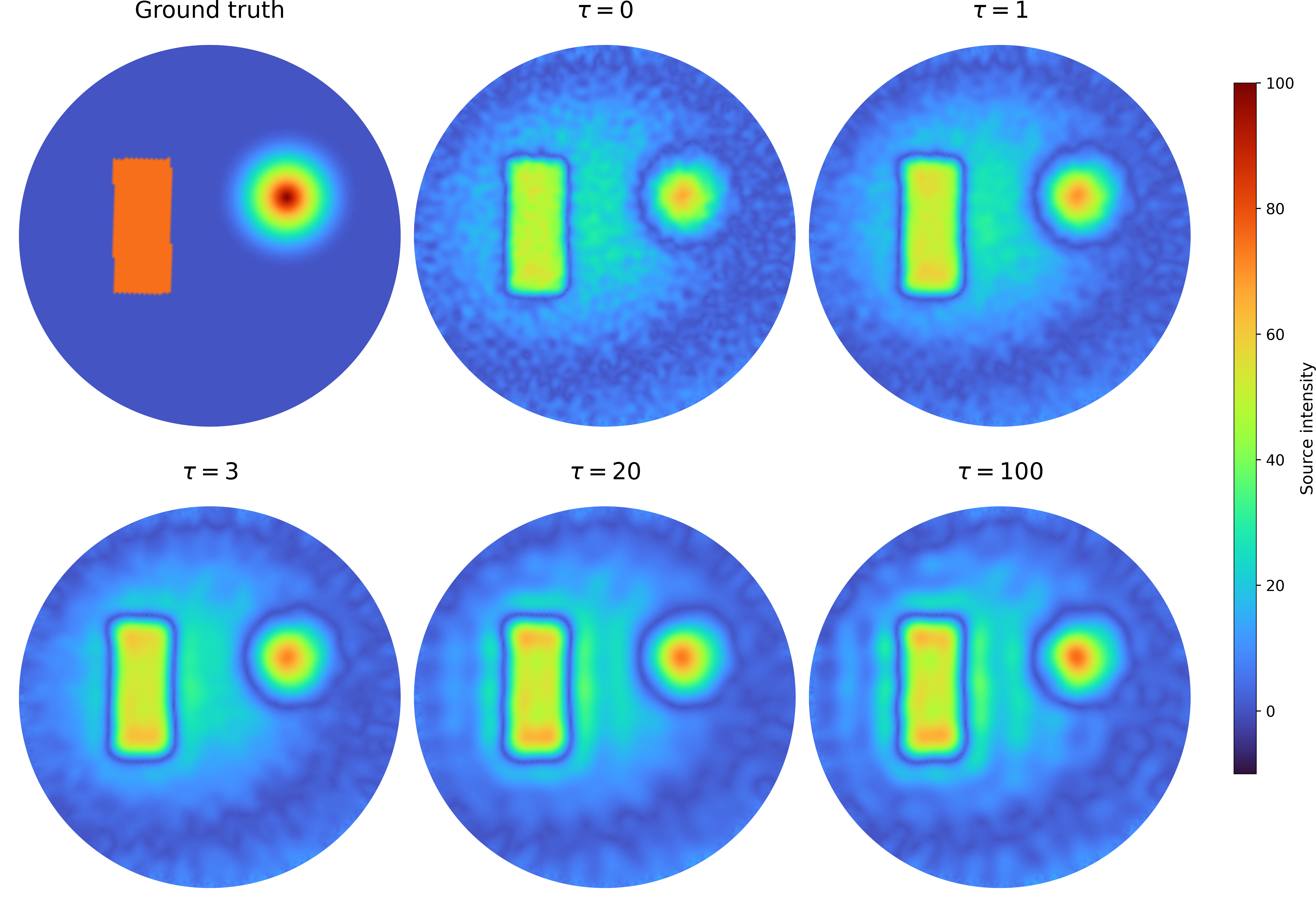}
    \caption{The ground truth (top left) and reconstructions for different values of $\tau$. We set $D$ to be the unit disc, and $D_0$ to be the disc centered at the origin with radius $0.99$. }
    \label{fig:111}
\end{figure}

\begin{figure}
    \centering
\includegraphics[width=0.75\linewidth]{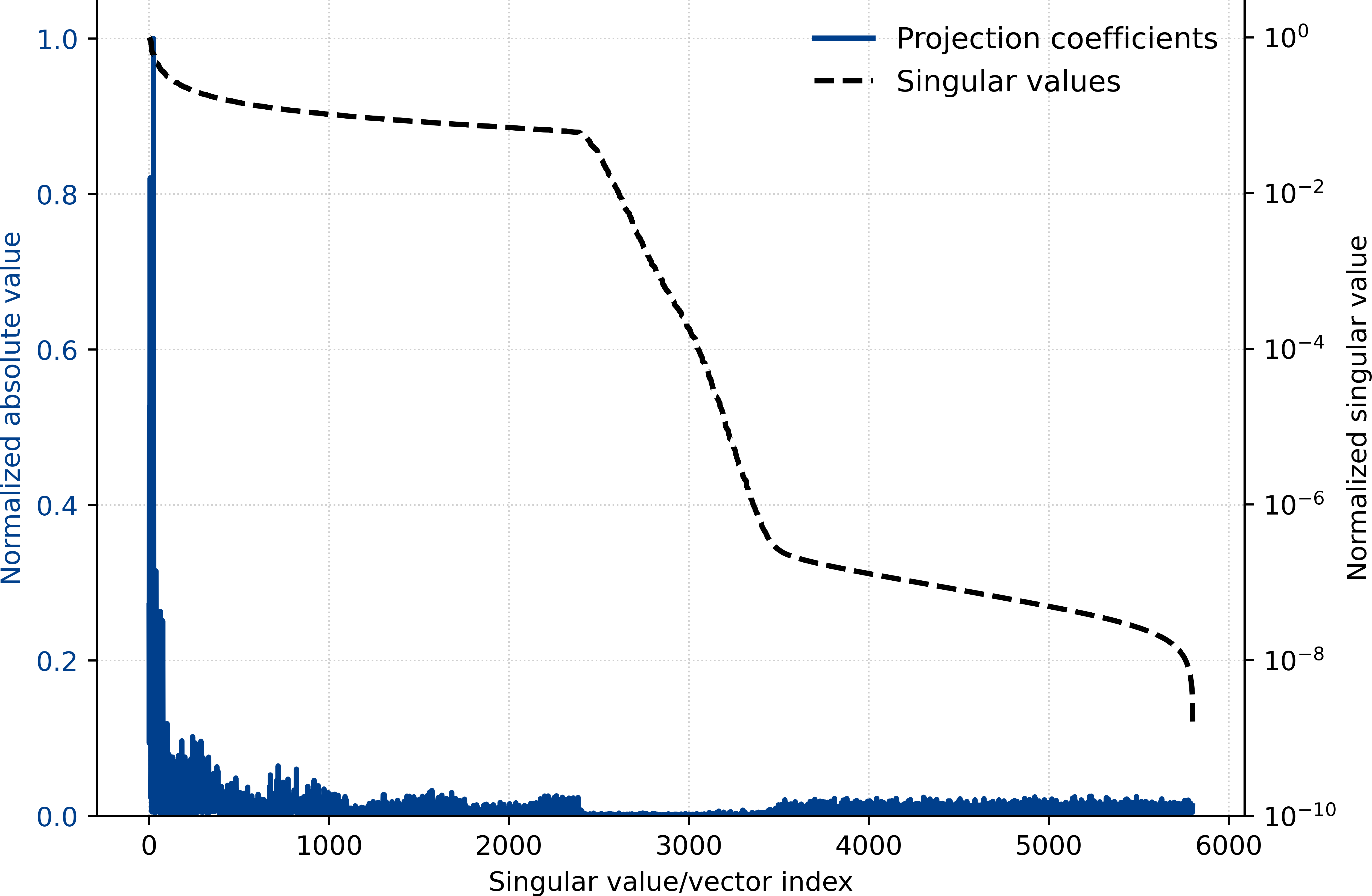}
\caption{The singular value spectrum of the combined forward operator, with projections of the ground truth onto the corresponding right singular vectors.}   \label{fig:222}
\end{figure}

%%%%%%%%%%%%%%%%%%%%%%%%%%%%%%%%%%%%%%%%%%%%%%%%%%%%%%%%%%%%%%%%%%%%

\appendix

%%%%%%%%%%%%%%%%%%%%%%%%%%%%%%%%%%%%%%%%%%%%%%%
\section{Appendix: Closedness of the operator~$H_\alpha$}
\label{closedness-H}

We may consider the operator $H=H_1$, since the positive
constant $(\sqrt{\alpha})^{1+\tau/2}$ doesn't play any role:
$$
H=V\left[\sigma^{-\tau/2}\right]V^*.
$$
Recall that $V\colon L^2(\M,\B,\mu)\to F$ is unitary.
If $\phi\in\ran H$, then
$H^{-1}\phi=V\left[\sigma^{\tau/2}\right]V^*\phi$.
It is readily seen that the domain of~$H$ is given by
$$
\D(H)=\setc{f\in F}{\sigma^{-\tau/2}V^*f\in L^2(\M,\B,\mu)}.
$$
Now assume that $f_n\in\D(H)$, that $f_n\to f$,
and $Hf_n\to\varphi$ in~$F$. We have:
\begin{eqnarray*}
\varphi=\lim Hf_n\;\hbox{in}\;F
&\Leftrightarrow&
V^*\varphi=\lim V^*Hf_n\;\hbox{in}\;L^2(\M,\B,\mu)\\
&\Leftrightarrow&
\big[\sigma(x)^{\tau/2}\big]V^*\varphi=
\lim \big[\sigma(x)^{\tau/2}\big]V^*Hf_n\;\hbox{in}\;L^2(\M,\B,\mu)\\
&\Leftrightarrow&
V\big[\sigma(x)^{\tau/2}\big]V^*\varphi=
\lim V\big[\sigma(x)^{\tau/2}\big]V^*Hf_n\;\hbox{in}\;F,
\end{eqnarray*}
in which the first and third equivalence stem
from the unitarity of~$V$,
while the second equivalence follows from the boundedness
of the function $\sigma(x)^{\tau/2}$.
Therefore, since $Hf_n\in\ran H$ and
$V\big[\sigma(x)^{\tau/2}\big]V^*=H^{-1}$, we obtain that
$f=H^{-1}\varphi$, which implies in turn that $f\in\D(H)$
and $Hf=\varphi$.
As a matter of fact,
\begin{eqnarray*}
\int\sigma(x)^{-\tau}\modc{(V^*f)(x)}^2\ud\mu(x)
&=&
\int\sigma(x)^{-\tau}\modc{V^*H^{-1}\varphi(x)}^2\ud\mu(x)\\
&=&
\int\sigma(x)^{-\tau}\sigma(x)^{\tau}
\modc{V^*\varphi(x)}^2\ud\mu(x)\\
&=&
\int\modc{V^*\varphi(x)}^2\ud\mu(x),
\end{eqnarray*}
the last integral being finite since
$V^*\varphi\in L^2(\M,\B,\mu)$.
This proves the closedness of~$H$.

\bibliographystyle{amsplain}
\bibliography{references}

\end{document}